\documentclass[a4paper]{amsart}

\usepackage{amscd,amsthm,amsfonts,amssymb,amsmath}
\usepackage[colorlinks=true]{hyperref}
\usepackage{fullpage}
\usepackage[all]{xy}

    \newcommand{\BA}{{\mathbb {A}}} 
    \newcommand{\BC}{{\mathbb {C}}} 
     
     \newcommand{\BH}{{\mathbb {H}}}

     \newcommand{\BN}{{\mathbb {N}}}
     
    \newcommand{\BQ}{{\mathbb {Q}}} \newcommand{\BR}{{\mathbb {R}}}

     \newcommand{\BZ}{{\mathbb {Z}}}

    \newcommand{\fa}{{\mathfrak{a}}} 
     
     \newcommand{\ff}{{\mathfrak{f}}}
     \newcommand{\fh}{{\mathfrak{h}}}

     \newcommand{\fp}{{\mathfrak{p}}}
    \newcommand{\fq}{{\mathfrak{q}}}

     \newcommand{\fz}{{\mathfrak{z}}}

    \newcommand{\ab}{{\mathrm{ab}}}

    \newcommand{\corank}{{\mathrm{corank}}}
    
    \newcommand{\Cl}{{\mathrm{Cl}}}

    \newcommand{\Gal}{{\mathrm{Gal}}} \newcommand{\GL}{{\mathrm{GL}}}

    \newcommand{\length}{{\mathrm{length}}}
    
    \newcommand{\loc}{{\mathrm{loc}}}

    \newcommand{\ord}{{\mathrm{ord}}} \newcommand{\rank}{{\mathrm{rank}}}

    \renewcommand{\mod}{\ \mathrm{mod}\ }

    \newcommand{\Sel}{{\mathrm{Sel}}} 
    
    \newcommand{\st}{{\mathrm{st}}}

    \font\cyr=wncyr10

    \newcommand{\Sha}{\hbox{\cyr X}}

    \newcommand{\ov}{\overline}

    \newcommand{\ra}{\rightarrow}

    \theoremstyle{plain}
    \newtheorem{thm}{Theorem}[section] 
      \newtheorem{prop}[thm]{Proposition}

\theoremstyle{remark} \newtheorem{remark}[thm]{Remark}
\theoremstyle{remark} 
\theoremstyle{remark} 

    \newcommand{\cO}{\mathcal O}
    \numberwithin{equation}{section}

\begin{document}
\title{A rank zero $p$-converse to a theorem of Gross--Zagier, Kolyvagin and Rubin}

\author{Ashay A. Burungale
and Ye Tian}

\address{Ashay A. Burungale:  California Institute of Technology,
1200 E California Blvd, Pasadena CA 91125, 
and The University of Texas at Austin, 
2515 Speedway, Austin TX 78712} 
\email{ashayburungale@gmail.com}

\address{Ye Tian:  School of Mathematical Sciences, University of Chinese Academy of Sciences, Beijing 10049, China.  MCM, HLM, Academy of Mathematics and Systems
Science, Chinese Academy of
Sciences, Beijing 100190, China} \email{ytian@math.ac.cn}

\maketitle
\begin{abstract} 
Let $E$ be a CM elliptic curve defined over $\BQ$ and $p$ a prime. We show that 
$$
\corank_{\BZ_{p}} \Sel_{p^{\infty}}(E_{/\BQ})=0 \implies \ord_{s=1}L(s,E_{/\BQ})=0
$$
for the $p^{\infty}$-Selmer group $\Sel_{p^{\infty}}(E_{/\BQ})$ and 
the complex $L$-function $L(s,E_{/\BQ})$. 
Along with Smith's work 
on the distribution of $2^\infty$-Selmer groups,
 this leads to the first instance 
of the even parity Goldfeld conjecture:  
For $50\%$ of the positive square-free integers $n$, we have 
$
\ord_{s=1}L(s,E^{(n)}_{/\BQ})=0, 
$
where $E^{(n)}: ny^{2}=x^{3}-x $ is a quadratic twist of the congruent number elliptic curve $E: y^{2}=x^{3}-x$.
\end{abstract}

\tableofcontents
\section{Introduction}
\noindent A $p$-converse to a theorem of Gross--Zagier, Kolyvagin and Rubin 
was recently proved \cite{BuTi} for CM elliptic curves in the rank one case for primes $p$ of good ordinary reduction. In this article we prove such a $p$-converse theorem for CM curves in the rank zero case for any prime $p$. 
As a consequence of the $2$-converse and Smith's work \cite{Sm2},
 we prove the first instance of the even parity Goldfeld conjecture. 

\subsubsection{$p$-converse} 
Let $E$ be an elliptic curve defined over $\BQ$. 
For a prime $p$, 
the $p^\infty$-Selmer group $\Sel_{p^{\infty}}(E_{/\BQ})$ encodes the arithmetic of $E$ 
via the exact sequence 
\label{Sel}
$$
0 \ra E(\BQ) \otimes_{\BZ} \BQ_{p}/\BZ_{p} \ra \Sel_{p^{\infty}}(E_{/\BQ}) \ra \Sha(E_{/\BQ})[p^{\infty}] \ra 0.
$$ 
Let $L(s,E_{/\BQ})$ be the associated Hasse--Weil $L$-function.
The Birch and Swinnerton-Dyer (BSD) conjecture connects the arithmetic of $E$ to the analytic properties of $L(s,E_{/\BQ})$. 



 The implication 
$$
\ord_{s=1}L(s,E_{/\BQ})=r \leq 1\implies \rank_{\BZ}E(\BQ)=r, \#\Sha(E_{/\BQ}) < \infty
$$
is a fundamental result towards the BSD conjecture due to Coates--Wiles \cite{CoWi}, Gross--Zagier \cite{GZ}, Kolyvagin \cite{Ko} and Rubin \cite{Ru0}. 
We refer to the implication 
\begin{equation}\label{p-cv}
\corank_{\BZ_{p}} \Sel_{p^{\infty}}(E_{/\BQ})=r \implies \ord_{s=1}L(s,E_{/\BQ})=r
\end{equation}
as the rank $r$ $p$-converse theorem. 
In the early 90's 
Rubin's seminal work \cite{Ru} led to the rank zero $p$-converse theorem for CM elliptic curves $E_{/\BQ}$ and primes 
\begin{equation}\label{sp}
p \nmid \# \cO_{K}^{\times}
\end{equation}
for $K$ the CM field.
More recently, the $p$-converse has been revived by Skinner \cite{Sk} and Zhang \cite{Zh}.

Our main result is the following.
\begin{thm}\label{ThmC}
Let $E$ be an elliptic curve defined over an imaginary quadratic field $K$, with complex multiplication by an order of $K$. Let $p$ be a prime.
Then 
$$
\corank_{\BZ_{p}}\Sel_{p^{\infty}}(E_{/K})=0 \implies \ord_{s=1}L(s,E_{/K})=0.
$$
In particular, if $E$ descends to $\BQ$, then \eqref{p-cv} holds. 
\end{thm}
In the main text the $p$-converse is proved for any CM newform (cf. Theorem \ref{pcv}). 

\subsubsection{Goldfeld's conjecture} 
Let $E$ be an elliptic curve defined over $\BQ$, and $E^{(n)}$ its quadratic twist associated to $n\in\BZ\setminus\{0\}$. In 1979 Goldfeld \cite{Gd} conjectured 
 that $\ord_{s=1}L(s,E^{(n)}_{/\BQ})=0$ for $50\%$ of the fundamental\footnote{Goldfeld also conjectured that $\ord_{s=1}L(s,E^{(n)}_{/\BQ})=1$ for $50\%$ of the fundamental discriminants $n$.} discriminants $n$ (the even parity Goldfeld conjecture). The first instance: 
\begin{thm}\label{ThmA}
The even parity Goldfeld conjecture is true for the congruent number elliptic curve 
$E:y^{2}=x^{3}-x$,  that is\footnote{Note that $\epsilon(E^{(n)})=+1$ precisely when $n \equiv 1,2,3 \mod{8}$ \cite{BS}.}, for a density one subset of the positive square-free integers $n \equiv 1,2,3 \mod{8}$,  
$$
\ord_{s=1}L(s,E^{(n)}_{/\BQ})=0.
$$ 
\end{thm}

  For perspectives on the 
  distribution of analytic ranks of elliptic curves, one may refer to Katz and Sarnak \cite[\S4]{KS}.
Under the GRH, Heath-Brown  \cite{HB2} proved the non-vanishing as in Theorem \ref{ThmA} for a positive proportion of the twists.
Smith \cite{Sm1} proved that 
$
\ord_{s=1}L(s,E^{(n)}_{/\BQ})=0
$
for at least $41.9\%$ of the positive square-free integers $n \equiv 1,2,3 \mod{8}$. The mod $2$ method \cite{Sm1} builds on  
\cite{Ti}, \cite{TYZ}, \cite{HB1}.

Our approach to Theorem \ref{ThmA} also yields the following. 
\begin{prop}\label{Gdiq-I}
Let $E$ be an elliptic curve defined over an imaginary quadratic field $K$, with CM by an order  
of $K$. Suppose that $3$ is not inert in $K$. 
Then for at least $50\%$ of $t \in K^{\times}/(K^{\times})^{2}$, 
$$
\ord_{s=1}L(s,E^{(t)}_{/K})=0.
$$
\end{prop}
\noindent This is an evidence towards Goldfeld's conjecture over imaginary quadratic fields \cite[Conj. 7.12]{KMR}, 
which posits $\ord_{s=1}L(s,E^{(t)}_{/K})=0$
for a density one subset of $t\in K^{\times}/(K^{\times})^{2}$
since $\epsilon(E^{(t)}_{/K})=+1$ for any $t$. 

\subsubsection{About the proofs} 
Theorem \ref{ThmA} is based on a $2$-adic method. It is a consequence of the rank zero $2$-converse Theorem \ref{ThmC}, and the distribution of $\corank_{\BZ_{2}}\Sel_{2^{\infty}}(E^{(n)}_{/\BQ})$ for positive square-free $n\equiv 1,2,3 \mod{8}$ due to Smith \cite{Sm2} (cf. Theorem \ref{ThmB}). Note that the prime $p=2$ is fundamental to \cite{Sm2}, and is excluded by the results of \cite{Ru}, \cite{Ru2}.

Our approach to the $p$-converse is based on Iwasawa theory of zeta elements  and is uniform for any prime $p$. 
For primes $p$ of non-ordinary reduction, the existence of a desirable $p$-adic $L$-function is still elusive. 
We consider main conjectures in terms of the $\GL_{1/K}$ and $\GL_{2/\BQ}$-zeta elements, namely elliptic units and Beilinson--Kato elements. 

Some of the following notation is not followed in the rest of the article. 
Let $T$ be the $p$-adic Tate module of $E_{/\BQ}$.
To recall Kato's main conjecture,
let $\BQ_{\infty}$ be the cyclotomic $\BZ_p$-extension of $\BQ$ and $\Lambda=\BZ_{p}[\![\Gal(\BQ_{\infty}/\BQ)]\!]$. 
Let $z_{E} \in H^{1}(\BQ,T)\otimes_{\BZ_{p}}\BQ_{p}$ be the associated Beilinson--Kato element.
For the Iwasawa cohomology $H^{1}(\BZ[\frac{1}{p}], T\otimes_{\BZ_{p}} \Lambda)$, let 
$${\bf{z}}_{E} \in H^{1}(\BZ[\frac{1}{p}], T\otimes_{\BZ_{p}} \Lambda) \otimes_{\BZ_{p}} \BQ_{p}$$ 
be the $\Lambda$-adic deformation of $z_{E}$ due to Kato.
Let $X_{\st}(E_{/\BQ_{\infty}})$ be the attached strict Selmer group.
Kato's main conjecture predicts: 
\begin{itemize}
\item[(i)] $X_{\st}(E_{/\BQ_{\infty}})$ is $\Lambda$-torsion and 
\item[(ii)] we have 
$$
\xi\big{(}H^{1}(\BZ[\frac{1}{p}], T \otimes_{\BZ_{p}} \Lambda) \otimes \BQ_{p}/ \Lambda \cdot{{\bf{z}}_{E}}\big{)}
=\xi(X_{\st}(E_{/\BQ_{\infty}})),
$$
an equality of ideals in $\Lambda \otimes_{\BZ_p} \BQ_{p}$ for $\xi(\cdot)$ the $\Lambda\otimes_{\BZ_p} \BQ_p$-characteristic ideal.
\end{itemize}
A key observation is that 
the rank zero $p$-converse 
is a consequence of Kato's main conjecture.

A principal result is a proof of the CM case of Kato's main conjecture (cf.~Theorem \ref{KMC}). 
It is approached   
via an equivariant main conjecture for the CM field involving elliptic units. This relies on 
a link between the Beilinson--Kato elements and elliptic units  
due to Kato \cite[\S15]{K}. 
As for the equivariant main conjecture up to tensoring with $\BQ_p$, a result of Johnson-Leung and Kings \cite{JLK} is employed.

Our approach to the $p$-converse differs from that of Rubin \cite{Ru}.
 It is uniform for any prime $p$, whereas \cite[\S11]{Ru} distinguishes the ordinary and non-ordinary primes, the latter being more complex. 
The Beilinson--Kato elements do not appear in \cite{Ru}, 
while they are elemental to the present approach. The hypothesis \eqref{sp} seems essential for the Euler system argument in \cite{Ru}.
The proof of Kato's main conjecture, Theorem \ref{KMC}, has other consequences. 
It gives rise to a rank one $2$-converse theorem for CM curves, which is utilised in the proof of a recent result of Alpoge--Bhargava--Shnidman: a positive proportion of the integers are the sum of two rational cubes. 



 
After tensoring with 
$\BQ_p$, the main conjectures bypass 
intricacies which remain to be resolved,  
especially for the prime $p=2$ (cf. \cite{CKLZ}).
For the rational version,
 the strategy of 
 \cite{JLK}
was inspired by Kato's pioneering work on the equivariant Tamagawa number conjecture (ETNC). 
We hope that Theorem \ref{ThmA} is just a glimpse of concrete applications of the general framework 
of the ETNC. 

It will be interesting to exhibit instances of the even parity Goldfeld conjecture for non-CM curves. 
Since the results of \cite{Sm2} apply, it is natural to seek the rank zero $2$-converse for non-CM curves. 
In the near future 
we plan to consider instances of Kato's main conjecture for $2$-ordinary non-CM curves.

\subsubsection{Organisation} Section \ref{s:KMC} presents the CM case of Kato's main conjecture.
We begin with a result towards the equivariant main conjecture for imaginary quadratic fields \S\ref{ss:EMC} and 
then prove Kato's main conjecture for CM newforms  
\S\ref{ss:KMC}.
Section \ref{s:pcv} describes the proofs of main results. 
The $p$-converse theorem \S\ref{ss:mr} is followed by results towards the Goldfeld conjecture \S\ref{ss:cp}.

\subsubsection*{Acknowledgement}
 We thank John Coates, Matthias Flach, Chandrashekhar Khare, Shinichi Kobayashi, Karl Rubin, Alex Smith  
and Wei Zhang for helpful conversations. It is a pleasure to thank Chris Skinner and Shou-Wu Zhang for inspiring discussions and encouragement. We are grateful to the referee for valuable suggestions. 

The debt this article owes to the work of Kazuya Kato is self-evident, and sincerely acknowledged.

\section{Kato's main conjecture}\label{s:KMC}
The principal result of this section is Theorem \ref{KMC}. 
\subsubsection{Notation} 
Let $c\in \Gal(\BC/\BR)$ be the complex conjugation.
 
Let $\ov{\BQ}$ be a fixed algebraic closure of $\BQ$. 
For a subfield $F \subset \ov{\BQ}$, let 
$G_{F}=\Gal(\ov{\BQ}/F)$. 
For a place $v$ of $F$, let $\ov{F}_{v}$ denote a fixed algebraic closure of $F_{v}$ and $G_{F_{v}}=\Gal(\ov{F}_{v}/F_{v})$. Choose an $F$-linear embedding $\ov{\BQ} \hookrightarrow \ov{F}_{v}$, which identifies $G_{F_{v}}$ as a subgroup of $G_{F}$. Let $p$ be a prime. Let $\mu_{p^{\infty}}$ be the set of $p$-power roots of unity in $\ov{\BQ}^\times$. Let $(\zeta_{p^{n}})_{n \geq 1}$ be a compatible system of primitive $p^{n}$-th roots of unity. 
Let $\epsilon:G_{\BQ} \ra \BZ_{p}^\times$ be the $p$-adic cyclotomic character. Fix embeddings $\iota_{\infty}: \ov{\BQ} \hookrightarrow \BC$ and $\iota_{p}: \ov{\BQ} \hookrightarrow \BC_{p}$.

For a ring $R$, a prime ideal $\fq \subset R$ and an $R$-module $M$, let $M_{\fq}$ denote the localisation. \subsection{Equivariant main conjecture}\label{ss:EMC}
We describe a result towards the equivariant main conjecture for imaginary quadratic fields. For details, one may refer to \cite{Ru}, \cite{JLK}, and especially \cite[\S15]{K}.

Let $K$ be an imaginary quadratic field. Fix an embedding $\iota: K \hookrightarrow \BC$. 
\subsubsection{Setup} 
Let $K^{\ab}$ be the maximal abelian extension of $K$ in $\BC$. For a non-zero ideal $\fa$ of $\cO_{K}$, 
let $K(\fa) \subset K^{\ab}$ be the ray class field of conductor $\fa$.

Fix a prime $p$ and a non-zero ideal $\ff$. Let 
$$
K(p^{\infty}\ff) = \bigcup_{n} K(p^{n}\ff), \text{  } G_{p^{\infty}\ff} = \Gal(K(p^{\infty}\ff)/K).
$$
Note that
$
G_{p^{\infty}\ff} \simeq \BZ_{p}^{2} \times \Delta_{\ff}
$
for $\Delta_{\ff}$ a finite group.

Let
$$
\fh^{1}=\varprojlim_{K'} (\cO_{K'}[1/p]^{\times} \otimes_{\BZ} \BZ_{p}) \text{, }
\fh^{2}=\varprojlim_{K'} \Cl_{K'}(p), 
$$
where $K'$ varies over finite extensions of $K$ contained in $K(p^{\infty}\ff)$, $\Cl_{K'}(p)$ denotes the $p$-primary part of the ideal class group $\Cl_{K'}$ and the inverse limit is with respect to the norm maps.
Recall that $\fh^{1}$ and $\fh^{2}$ are finitely generated modules over the three-dimensional semi-local ring 
$\BZ_{p}[\![G_{p^{\infty}\ff}]\!]$. Moreover, $\fh^{1}$ is a torsion-free $\BZ_{p}[\![G_{p^{\infty}\ff}]\!]$-module such that
$\fh^{1}_{\fq}$ is of dimension one for any height zero prime ideal $\fq$ of $\BZ_{p}[\![G_{p^{\infty}\ff}]\!]$,  
and $\fh^{2}$ is a torsion $\BZ_{p}[\![G_{p^{\infty}\ff}]\!]$-module.
Let
$$
\fz \subset \fh^{1}
$$
be the submodule of elliptic units as in \cite[\S15.5]{K}. Then $\fh^{1}/\fz$ is a torsion 
$\BZ_{p}[\![G_{p^{\infty}\ff}]\!]$-module.

For a $\BZ_{p}[\![G_{p^{\infty}\ff}]\!]$-module $M$, put $M_{\BQ}=M \otimes_{\BZ} \BQ$.
For a $(\BZ_{p}[\![G_{p^{\infty}\ff}]\!] \otimes_{\BZ}\BQ)$-torsion module $N$, let $\xi(N)$ be the characteristic ideal.

\subsubsection{Equivariant main conjecture} The following result towards the equivariant main conjecture for imaginary quadratic fields is due to Johnson-Leung and Kings \cite{JLK}. It builds on \cite{Ru} and \cite{HK}. 
\begin{thm}\label{EMC}
Let the notation be as above. 
Then we have
$$
\xi (\fh^{2}_{\BQ}) = 
\xi ((\fh^{1}/\fz)_{\BQ}),
$$
an equality of ideals in $\BZ_{p}[\![G_{p^{\infty}\ff}]\!] \otimes \BQ$.
\end{thm}
\begin{proof}
Let $\fp$ be a height one prime ideal of $\BZ_{p}[\![G_{p^{\infty}\ff}]\!]$ such that $p \notin \fp$.
Then we have
$$
\length_{\BZ_{p}[\![G_{p^{\infty}\ff}]\!]_{\fp}} \fh^{2}_{\fp} = 
\length_{\BZ_{p}[\![G_{p^{\infty}\ff}]\!]_{\fp}} ((\fh^{1}/\fz)_{\fp}), 
$$
after\footnote{One may to \cite[\S15.6]{K} for a description of the underlying Iwasawa modules in terms of Iwasawa cohomology. The formulation in Theorem \ref{EMC} is then identical with the one in \cite[\S7.2]{JLK}.} Johnson-Leung and Kings
\cite[\S7.2]{JLK}. 
\end{proof}
\begin{remark}
For height one primes $\fp$ of $\BZ_{p}[\![G_{p^{\infty}\ff}]\!]$ 
with $p \in \fp$, a result towards the equivariant main conjecture is proved in \cite[Thm. 5.7 \& \S5.5]{JLK}. 
The conjecture is still open for $p$ non-split in $K$. 
\end{remark}

\subsection{Kato's main conjecture}\label{ss:KMC}
For background, one may refer to \cite[\S12 \& \S15]{K}.

\subsubsection{Elliptic newforms} 
Let $f \in S_{k}(\Gamma_{1}(N))$ be an elliptic newform with $q$-expansion coefficents $\{a_{n}(f)\}_{n\in \BN}$ and 
$F=\BQ(\{a_{n}(f)\}_{n})$ the Hecke field.  

For a place $\lambda$ of $F$ above $p$, let $F_{\lambda}$ be the completion and $O_{\lambda}$ its integer ring. Let $V_{F_{\lambda}}(f)$ be the two-dimensional $G_{\BQ}$-representation 
associated to $f$ as in \cite[(8.3)]{K}.

\subsubsection{Cyclotomic tower}\label{ss:cyc}
For $n \in \BZ_{\geq 0}$, let
$$
G_{n}=\Gal(\BQ(\zeta_{p^{n}})/\BQ)\text{, } 
G_{\infty}=\varprojlim_{n} G_{n}=\Gal(\BQ(\zeta_{p^{\infty}})/\BQ).
$$
The cyclotomic character induces an isomorphism $\kappa: G_{\infty} \simeq \BZ_{p}^\times$, and let $\sigma_c$ be the inverse image of $c\in\BZ_p^\times$. 
Recall that $\BZ_{p}^{\times}\simeq\BZ_{p} \times \Delta$ for
$$
\Delta \simeq
\begin{cases}
\BZ/(p-1) & \text{if $p$ odd} \\
\BZ/2 & \text{else}.
\end{cases}
$$

Let
$$
\Lambda=O_{\lambda}[\![G_{\infty}]\!]
$$
be a two dimensional complete semi-local ring.

\subsubsection{Iwasawa cohomology} Let $T \subset V_{F_{\lambda}}(f)$ be a $G_{\BQ}$-stable $O_{\lambda}$-lattice. For $q \in \BZ$, let
$$
\BH^{q}(T) = \varprojlim_{n} H^{q}(\BZ[\zeta_{p^{n}},1/p],T)
$$
be a $\Lambda$-module, where $H^{q}$ denotes the \'etale cohomology as in \cite[8.2]{K}. 
Let $$\BH^{q}(V_{F_{\lambda}}(f))=\BH^{q}(T)\otimes \BQ$$ be a $(\Lambda \otimes \BQ)$-module. 

The following basic result is due to Kato \cite[Thm 12.4]{K}. 
\begin{thm}\label{Ka1}
Let $f \in S_{k}(\Gamma_{1}(N))$ be an elliptic newform and $p$ a prime. 
Then: 
\begin{itemize}
\item[(1)] $\BH^{2}(V_{F_{\lambda}}(f))$ is a torsion $(\Lambda \otimes \BQ)$-module and
\item[(2)] $\BH^{1}(V_{F_{\lambda}}(f))$ is a free $(\Lambda \otimes \BQ)$-module of rank one.
\end{itemize}
\end{thm}

\subsubsection{Beilinson--Kato elements}
The following pivotal existence of zeta elements for the $p$-adic Galois representation associated to an elliptic newform is due to Kato \cite[Thm. 12.5]{K}. 

\begin{thm}\label{BK}
Let $f \in S_{k}(\Gamma_{1}(N))$ be an elliptic newform and $p$ a prime. 
\begin{itemize}
\item[(1)] There exists a non-zero $F_{\lambda}$-linear map
$$
V_{F_{\lambda}}(f) \ra \BH^{1}(V_{F_{\lambda}}(f)); 
\text{  $\gamma \mapsto {\bf{z}}_{\gamma}(f)$}.
$$
\item[(2)] Let $Z(f)$ be the $(\Lambda \otimes \BQ)$-submodule of $\BH^{1}(V_{F_{\lambda}}(f))$ 
generated by ${\bf{z}}_{\gamma}(f)$ for all $\gamma \in V_{F_{\lambda}}(f)$. Then 
$$
\BH^{1}(V_{F_{\lambda}}(f))/Z(f)
$$
is a torsion $(\Lambda \otimes \BQ)$-module.
\end{itemize}
\end{thm}
\begin{remark} For a characterising property of the map $\gamma \mapsto {\bf{z}}_{\gamma}(f)$ in terms of the underlying critical $L$-values,  
one may refer to \cite[Thm. 12.5 (1)]{K}.
\end{remark}


\subsubsection{CM modular forms} 
A Hecke character $\psi$ over $K$ has infinity type $(m,n)$ if the restriction of $\psi$ to 
the archimedean part $\BC^{\times}$ of the idele class group over $K$ is of the form
$$
z \mapsto z^{m} \cdot \bar{z}^{n}, 
$$
where $\bar{\cdot} \in \Gal(K/\BQ)$ is induced by the embedding $\iota$ and the complex conjugation $c$.

Let $f \in S_{k}(\Gamma_{1}(N))$ be a CM newform.
Then, we have 
$$
L(s,f)=L(s,\psi)
$$
for a Hecke character $\psi$ over the underlying CM field $K$ with infinity type $(1-k,0)$, 
 and $L(s,\cdot)$ the complex $L$-function.
 We follow arithmetic normalisation for the complex $L$-functions. 

\subsubsection{Kato's main conjecture} 
\begin{thm}\label{KMC}
Let $f \in S_{k}(\Gamma_{1}(N))$ be a CM newform and $p$ a prime. 
Let $F$ be the Hecke field and $\lambda$ a prime above $p$. 
Let $V_{F_{\lambda}}(f)$ be the associated Galois representation and 
$\BH^{q}(V_{F_{\lambda}}(f))$ the Iwasawa cohomology.
Let $Z(f) \subset \BH^{1}(V_{F_{\lambda}}(f))$ be the $(\Lambda \otimes \BQ)$-submodule 
of Beilinson--Kato elements.
Then 
$$
\xi(\BH^{2}(V_{F_{\lambda}}(f)))=\xi(\BH^{1}(V_{F_{\lambda}}(f))/Z(f)), 
$$
an equality of ideals in $\Lambda \otimes \BQ$.
\end{thm}
\begin{proof}
In \cite[\S15]{K} it is shown that the CM case of Kato's main conjecture \cite[Conj. 12.10]{K} in $\Lambda \otimes \BQ$ is a consequence\footnote{Being in a setup of $(\Lambda \otimes \BQ)$-modules, the hypotheses of \cite[Lem. 15.13]{K} hold.} of the equivariant main conjecture for the underlying imaginary quadratic field up to tensoring with $\BQ$. 
(Especially see \cite[15.6]{K}.) 
This is based on 
a link 
between the Beilinson--Kato elements and elliptic units \cite[(15.16.1)]{K}.
The proven case  \cite[Thm. 15.2]{K} of the equivariant main conjecture due to Rubin \cite{Ru} is employed in \cite[Prop. 15.17]{K}. 
In the present setting we instead utilise Theorem \ref{EMC}.
\end{proof}
\begin{remark}
A finer analysis of \cite[\S15]{K} is essential to relate the integral versions of the main conjectures in the above proof, 
which we hope to study in the near future.  
\end{remark}
\section{Proofs of main results}\label{s:pcv}
The central result is Theorem \ref{pcv}, which leads to Theorem \ref{ThmC}, Theorem \ref{ThmA} and Proposition \ref{Gdiq-I}.

\subsection{$p$-converse}\label{ss:mr} 
\begin{thm}\label{pcv}
Let $f \in S_{k}(\Gamma_{1}(N))$ be a CM newform of even weight and $p$ a prime. 
Let $F$ be the Hecke field and $\lambda$ a prime above $p$. 
Let $V_{F_{\lambda}}(f)$ be the associated 
Galois representation
and 
$H^{1}_{\rm{f}}(\BQ,V_{F_{\lambda}}(f)(k/2))$ the Bloch--Kato Selmer group.
Then 
$$
H^{1}_{\rm{f}}(\BQ,V_{F_{\lambda}}(f)(k/2))=0 \implies \ord_{s=k/2}L(s,f)= 0.
$$
\end{thm}
\begin{proof}
The following is based on Theorem \ref{KMC}.

Let $T \subset V_{F_{\lambda}}(f)$ be a $G_{\BQ}$-stable $O_{\lambda}$-lattice. 
Since $H^{1}_{\rm{f}}(\BQ,V_{F_{\lambda}}(f)(k/2))=0$, by the exact sequence \cite[(14.9.3)]{K}, 
note that 
\begin{equation}\label{van}
H^{2}(\BZ[1/p], T(k/2)) \otimes \BQ=0.
\end{equation}
Let $\pm = (-1)^{k/2 - 1}$ and pick $\gamma \in V_{F_{\lambda}}(f)$ such that $\gamma^{\pm} \neq 0$. Let $z$ be the image of 
the Beilinson--Kato element ${\bf{z}}_{\gamma}(f)$ (cf. Theorem \ref{BK}) under the map  
\begin{equation}\label{z-sp}
\BH^{1}(V_{F_{\lambda}}(f)) \simeq \BH^{1}(V_{F_{\lambda}}(f)(k/2)) \ra H^{1}(\BZ[1/p], T(k/2)) \otimes \BQ.
\end{equation}

By \eqref{van} and the Euler--Poincar\'e formula of Tate \cite[(14.9.5)]{K}, 
$$
\dim_{F_{\lambda}}H^{1}(\BZ[1/p],T(k/2)) \otimes \BQ=1. 
$$
Let $\fq$ be the kernel of the $O_\lambda$-homomorphism $\Lambda \ra O_\lambda$ mapping  
$\sigma_{c}$ to $c^{-k/2}$ for any $c\in\BZ_p^\times$ (cf. \S\ref{ss:cyc}). 
Note that the map \eqref{z-sp} factors through 
$$\BH^{1}(V_{F_{\lambda}}(f))_{\fq}/\fq \BH^{1}(V_{F_{\lambda}}(f))_{\fq},$$ 
and $\BH^{2}(V_{F_{\lambda}}(f))_{\fq}$ vanishes by \eqref{van}. 
Hence Theorem \ref{KMC} implies that 
 $z$ is an $F_{\lambda}$-basis of 
$H^{1}(\BZ[1/p],T(k/2)) \otimes \BQ$. (See also \cite[p. 242]{K}.) 

Since $H^{1}_{\rm{f}}(\BQ,V_{F_{\lambda}}(f)(k/2))$ vanishes, it follows that 
\begin{equation}\label{z-loc}
\loc_{p}(z) \notin H^{1}_{\rm{f}}(\BQ_{p},V_{F_{\lambda}}(f)(k/2))
\end{equation} 
for $\loc_{p}: H^{1}(\BQ,V_{F_{\lambda}}(f)(k/2)) \ra H^{1}(\BQ_{p},V_{F_{\lambda}}(f)(k/2))$ the localisation and $H^{1}_{\rm{f}}(\BQ_{p},V_{F_{\lambda}}(f)(k/2))$ the Bloch--Kato subgroup. 
Finally, \eqref{z-loc} and Kato's explicit reciprocity law \cite[Thm. 12.5 (1)]{K} imply that
$$
\ord_{s=k/2} L(s,f) = 0.
$$
\end{proof}
\begin{remark}
This deduction of the rank zero $p$-converse 
from Kato's main conjecture holds for any elliptic newform.
\end{remark}
\vskip2mm
{\it{Proof of Theorem \ref{ThmC}}}.
Let $\psi: \BA_{K}^{\times}/K^{\times} \ra K^\times$ be the Hecke character associated to $E_{/K}$. Let $f$ and 
$\ov{f}$ be the CM newforms associated to $\psi$ and $\ov{\psi}$, the corresponding Hecke field being a subfield of $K$. Note that $L(s,E_{/K})=L(s,\psi)\cdot L(s,\ov{\psi})=L(s,f)\cdot L(s,\ov{f})$. 
In the same vein one has  
$$\Sel_{p^{\infty}}(E_{/K})^{\vee}\otimes K \simeq 
 H^{1}_{\rm{f}}(\BQ,V_{K_{\lambda}}(f)) 
\oplus  H^{1}_{\rm{f}}(\BQ,V_{K_{{\lambda}}}(\ov{f})) 
 ,$$ 
where 
the superscript `$\vee$' denotes the Pontryagin dual, and 
$V_{K_{\lambda}}(\cdot)=V_{F_{\lambda}}(\cdot)\otimes_{F}K$.
Hence, by Theorem \ref{pcv}, 
$$
\corank_{\BZ_{p}}\Sel_{p^{\infty}}(E_{/K})=0 \implies \ord_{s=1}L(s,E_{/K})=0.
$$

If $E$ descends to $\BQ$, then $\corank_{\BZ_{p}}\Sel_{p^{\infty}}(E_{/K})=2\cdot \corank_{\BZ_{p}}\Sel_{p^{\infty}}(E_{/\BQ})$ and $\ord_{s=1}L(s,E_{/K})=2\cdot\ord_{s=1}L(s,E_{/\BQ})$, which concludes the proof of Theorem \ref {ThmC}. 
\subsection{Goldfeld's conjecture}\label{ss:cp} 
\subsubsection{Over the rationals} The proof of 
Theorem \ref{ThmA} is 
partly based on the following progress towards the distribution of the associated $2^{\infty}$-Selmer ranks due to Smith \cite[Thm.~1.2]{Sm2}.

\begin{thm}\label{ThmB}
Let $E:y^{2}=x^{3}-x$ be the congruent number elliptic curve.
Then for a density one subset of the square-free positive integers $n \equiv 1,2,3 \mod{8}$, 
 $$
\corank_{\BZ_{2}}\Sel_{2^{\infty}}(E^{(n)}_{/\BQ})=0.
$$
\end{thm}
\begin{remark}
For earlier results, one may refer to \cite{Ti}, \cite{TYZ}, \cite{Sm1} and references therein. 
\end{remark}

\vskip2mm
{\it{Proof of Theorem \ref{ThmA}}}. This is a consequence of Theorem \ref{ThmB} and the $2$-converse Theorem \ref{ThmC}. 

\subsubsection{Over imaginary quadratic fields} 
The proof of Proposition \ref{Gdiq-I} 
utilises the following 
\cite[Thm. 2.7]{BKLS}. 

\begin{thm}\label{Gdiq}
Let $E$ be an elliptic curve defined over an imaginary quadratic field $K$, with CM by an order  
of $K$. Suppose that $3$ is not inert in $K$. 
Then for at least $50\%$ of $t \in K^{\times}/(K^{\times})^{2}$, 
$$
\corank_{\BZ_{3}}\Sel_{3^{\infty}}(E^{(t)}_{/K})=0.
$$
\end{thm}
\vskip2mm
{\it{Proof of Proposition \ref{Gdiq-I}}}. This is a consequence of Theorem \ref{Gdiq} and the $3$-converse Theorem \ref{ThmC}. 
\vskip2mm 
\begin{remark}
\begin{itemize}
\item[(i)] For any $n\in \BZ\setminus\{0\}$, the cube sum elliptic curve $E_{n}: x^{3}+y^{3}=2n$ 
 satisfies the hypotheses of Proposition \ref{Gdiq-I}, with $K=\BQ(\zeta_3)$. 
\item[(ii)] A positive proportion of the quadratic twists in Proposition \ref{Gdiq-I} do not descend to $\BQ$.
The proposition yields (perhaps) the first examples of elliptic curves $E$ 
over a number field $F$ different from $\BQ$ for which a weak form of the Goldfeld conjecture holds: the analytic rank being zero for a positive proportion of the quadratic twists of $E$ over $F$ which do not descend to $\BQ$.
\end{itemize}
\end{remark}

\end{document}